\documentclass[11pt]{amsart} 

\usepackage{amsmath,amsfonts,amssymb,mathtools,amsthm}
\usepackage{mathrsfs}

\theoremstyle{plain}

\newtheorem{theorem}{Theorem}[section]
\newtheorem{lemma}[theorem]{Lemma}

\newtheorem{corollary}[theorem]{Corollary}

\newtheorem{proposition}[theorem]{Proposition}

\theoremstyle{definition}

\newtheorem{hypothesis}[theorem]{Hypothesis}
\newtheorem{remark}[theorem]{Remark}

\numberwithin{equation}{subsection}

\newcommand{\cc}{{\mathbb C}}
\newcommand{\pp}{{\mathbb P}}
\newcommand{\rr}{{\mathbb R}}

\newcommand{\zz}{{\mathbb Z}}

\newcommand{\Gm}{{\mathbb G}_{\rm m}}
\newcommand{\aaaa}{{\mathbb A}}

\newcommand{\Oo}{{\mathcal O}}

\newcommand{\srG}{{\mathscr G}}

\title[Twistor space for local systems]{Twistor space for local systems on an open curve}
\author{Carlos Simpson}
\date{} 

\keywords{Connection, Formal neighborhood, Groupoid, Harmonic bundle, Hecke transformation, Higgs bundle, Logarithmic differential,  Moduli space,  Representation, Residue, Riemann-Hilbert correspondence, Twistor space}

\subjclass[2010]{Primary 14D21; Secondary  14F35, 22A22, 32J25}

\begin{document}

\begin{abstract}
Let $X=\overline{X}-D$ be a smooth quasi-projective curve. In \cite{tgps} we constructed
a Deligne-Hitchin modui space with Hecke gauge groupoid for connections of rank $2$. 
We extend this construction to the case of any rank $r$, although still keeping 
a genericity hypothesis. The formal neighborhood of a preferred section
corresponding to a tame harmonic bundle is governed by a mixed twistor structure. 
\end{abstract}

\maketitle

\section{Introduction}

Let $X$ be a smooth quasi-projective curve with basepoint $x\in X$.  Consider
the  smooth compactification
$j: X\hookrightarrow \overline{X}$ and set $D:= \overline{X} -X$.  Denote the points of $D$
by $t_1,\ldots , t_k$. A solution of Hitchin's equations over $X$
may be viewed as a harmonic bundle 
$(E,\partial , \overline{\partial}, \varphi , \overline{\partial},h )$ and we assume
it is tame: the eigenvalues of $\varphi$ have at most first-order poles at the $t_i$. 
This yields
a parabolic Higgs bundle, and on the other hand a
vector bundle with flat connection that is also furnished with a parabolic
structure. More generally, for any complex parameter 
$\lambda \in \aaaa^1$ we obtain a vector bundle $F^{\lambda}$ with 
parabolic structure and a $\lambda$-connection $\nabla ^{\lambda}$ compatible
with the parabolic structure. As $\lambda$ varies, these objects vary holomorphically,
except for the variation of parabolic levels (i.e. weights) that is real analytic. 

Recall 
that the space $\aaaa^1$ of values of $\lambda$ 
extends to $\pp^1$, the twistor line \cite{Hitchin,HKLR,twistor,Sabbah,Mochizuki,weight2,
AlfayaGomez,BiswasHellerRoser,BeckHellerRoser,tgps}. 
We would like to create a family of moduli spaces over $\pp^1$ such that
the collection of $(F^{\lambda}_{\cdot } , \nabla ^{\lambda})$, extended to
the chart at $\infty$ by the Deligne glueing, provides a {\em preferred section} corresponding
to the harmonic bundle. 

This should generalize Hitchin's original construction \cite{Hitchin}, 
with a twistor-like property to be formulated
in terms of the notion of mixed twistor structure combining weights $1$ and $2$, also with a component
of weight $0$ for the variation in choice of framing. 

The parabolic levels of $E^{\lambda}_{\cdot }$ and the eigenvalues of the
residue of $\nabla^{\lambda}$ vary in $r$ families 
$({\mathfrak p}_{\lambda},{\mathfrak e}_{\lambda})$ 
governed by the 
Sabbah-Mochizuki formulas 
\eqref{frakep}. In view of the formula
for ${\mathfrak p}$, we see that the ordering of the parabolic levels is not invariant
as $\lambda$ moves in $\aaaa^1$. On the other hand, from 
\eqref{frakep} for ${\mathfrak e}$ there also exist 
points where two eigenvalues
come together, and there we do need the filtrations of the 
parabolic structure in order to define a good moduli problem. The reader may refer to Sabbah's picture 
in \cite{Sabbah}. 

In \cite{tgps} we explained how to get around this problem in the case of bundles of rank $2$. 
The idea was to notice that the two types of problems: when parabolic levels cross over,
and where residual eigenvalues come together, don't occur at the same time, if we
make a global distinctness assumption on the KMS spectrum elements. 
At points where the eigenvalues are distinct, we can freely change the order of the eigenspaces in the
parabolic filtration. We therefore defined the moduli space as a quotient by the {\em Hecke gauge groupoid} allowing permutation of orderings of the filtration steps, at points
where the eigenvalues are distinct. 

A general discussion of the idea that there is a weight $2$ nonabelian Hodge structure
governing the KMS spectrum was done in \cite{weight2} for the case of rank $1$, and 
in \cite{tgps} for the case of rank $2$. We refer to \cite{tgps}
for general remarks about this point of view, 
as well as for various aspects that don't need to be
repeated here. 

The goal of this paper is to explain how to 
extend these considerations from rank $\leq 2$  to the case of bundles of
arbitrary rank. For this, we need to provide a precise definition of the moduli spaces
involved---relatively straightforward---and of the appropriate Hecke gauge groupoid. 
In the rank $2$ case a simplified description of the gauge groupoid was made possible
because there were only two eigenvalues to take care of. In the present case, the basic 
idea of the definition is again to say that two eigenspaces may be permuted in the ordering of the
filtrations, whenever they have distinct residual eigenvalues. We need to check that the
resulting definition is coherent when several such operations are composed. 

\begin{theorem}
\label{maintheorem}
There is a smooth but non-separated analytic groupoid 
$$
( \widetilde{M}_{\rm DH}, \, \srG _{\rm DH} )
\rightarrow \pp^1
$$
whose chart
over the standard $\aaaa^1\subset \pp^1$
parametrizes rank $r$ quasi-parabolic bundles with logarithmic $\lambda$-connection
on $(\overline{X},D)$ and framing at the basepoint $x\in X$, up to
local gauge transformations over $D$; the other chart is  
the same kind of space for the complex conjugate curve. A tame harmonic bundle on $X$ 
satisfying certain genericity hypotheses yields
a well-defined preferred section $\rho$. In turn, $\rho$ has a separated smooth analytic
neighborhood. The family of complete local rings of 
the formal completion of this neighborhood along $\rho$ is a mixed twistor structure,
such that the mixed twistor structure on the normal bundle has weights $0,1,2$. 
\end{theorem}

We use moduli spaces of framed local systems and
the weight $0$ piece of the normal bundle corresponds to
changes of framing. 
The weight $2$ piece
corresponds to the variation of KMS spectrum. We refer to 
\cite{weight2,tgps} for a discussion of the relationship between the 
the weight $2$ property and the formulas \eqref{frakep} for
${\mathfrak p}$ and ${\mathfrak e}$.
The weight $1$ piece corresponds to having hyperkähler structures
\cite{Nakajima,Holdt}
on symplectic
leaves of the moduli space. 

The mixed
twistor property for the full formal completion around a preferred section
generalizes, to a certain extent, the
mixed Hodge structure on the completion of the space of representations at
a VHS \cite{EyssidieuxSimpson,EKPR,Pridham, Lefevre1,Lefevre2,GreenGriffithsKatzarkov}. However, 
we restrict to $X$ being a curve, and the genericity hypotheses in effect lead to a
smooth moduli space, so the local structure here is simpler
than in the previous references.

It is a great pleasure to dedicate this
paper to Oscar García-Prada. Exploring the world of gauge theory, representations of
fundamental groups, Higgs bundles and the like, the influence of Oscar's ideas is always felt. 
His contributions have notably opened up many 
potential future perspectives where Hitchin's twistor viewpoint
started in \cite{Hitchin,HKLR} should be relevant. These include
cohomology calculations for parabolic moduli spaces 
\cite{OGPGothenMunoz}, relations with Lie theory for real and complex
groups \cite{BiquardOGPMiR,OGPLogaresMunoz}, the Cayley correspondence
\cite{CayleyCorr}, actions of mapping class
groups \cite{OGPWilkin} and a fascinating but mostly untouched
theory of parabolic vortex equations \cite{ACOGP,BiquardOGP}. It is hoped that the
present paper will provide some tools for the development of twistor theory
in these directions.

\section{Three structures}

We'll be talking about three different but related types of structures: tame harmonic bundles,
parabolic logarithmic $\lambda$-connections, and quasi-parabolic logarithmic 
$\lambda$-connections. 
Recall that (quasi-)parabolic structures were introduced by Seshadri  \cite{Seshadri}
with progress by many authors over the subsequent decades. 

We make the convention that quasi-parabolic objects have filtrations that
are {\em complete}, i.e. the successive quotients have dimension $1$, whereas parabolic objects can have
more general filtrations. 

Starting from a tame harmonic bundle we obtain for any $\lambda \in \aaaa^1$ a
parabolic logarithmic $\lambda$-connection. If $\lambda = 0$ this is the Higgs bundle and
if $\lambda = 1$ it is the flat bundle. One problem with this construction is that it doesn't
vary nicely as a function of $\lambda$. The moduli space theory would be more difficult to develop
too, because of the higher-rank quotients in the filtrations and therefore the possibility of residues
with non-trivial nilpotent parts. We therefore pass to quasi-parabolic structures for the moduli problem,
making a genericity hypothesis \ref{hypHB} that the KMS spectrum elements occur with
multiplicity $1$, to insure that there is a well-controlled way of
extracting a quasi-parabolic structure. Investigation of this process and the corresponding equaivalence
relation that needs to be put on the moduli spaces, will be the main part of this work. 

\subsection{Quasi-parabolic structures}
\label{qp-struct}

Consider the category $QPL(X,\lambda ,r)$ consisting of objects
$(\lambda , F, \nabla , V )$ where:
\begin{itemize}

\item 
$F$ be a vector bundle of rank $r$ on $\overline{X}$;

\item 
$\nabla : F \rightarrow F\otimes \Omega ^1_Y(\log D)$ is a 
logarithmic $\lambda$-connection, i.e. an operator
such that $\nabla (ae) = a\nabla (e) + \lambda d(a)e$, in particular
at each point $t\in D$ 
we obtain the {\em residue}
${\rm res}_t(\nabla ) : F_t \rightarrow F_t$;

\item
$V$ is a collection of filtrations for each $t\in D$ 
$$
0 = V_{t,0} \subset V_{t,1} \subset \cdots \subset V_{t,r} = F_t
$$
such that the quotients ${\rm gr}_{t,j}:= V_{t,j} / V_{t,j-1}$, are $1$-dimensional,  and such that 
$$
{\rm res}_t(\nabla ) : V_{t,j} \rightarrow V_{t,j}.
$$
\end{itemize}
These will be called {\em quasi-parabolic 
logarithmic $\lambda$-connections of rank $r$}
\cite{BhosleRamanathan,Biswas,BiswasInabaKomyoSaito,BodenYokogawa,
InabaIwasakiSaito,Konno,MaruyamaYokogawa,MehtaSeshadri,Nakajima,
NitsureLog,Seshadri,Singh}.

Morphisms are those which preserve all the structures. 
A {\em framed object} is an object of the category provided with
a framing $\beta : F_x \stackrel{\cong}{\rightarrow} \cc^r$. 

We obtain the vector of residual eigenvalues at $t$ by letting  
$\theta _{t,j}(F,\nabla , V_{\cdot} )\in \cc$ 
be the scalar giving the action of ${\rm res}_t(\nabla )$ on 
${\rm gr}_{t,j}$. 

The constructions \cite{BalajiSeshadri, BhosleRamanathan,BiquardOGPMiR,BodenYokogawa,
OGPLogaresMunoz,
Konno,MaruyamaYokogawa}
of moduli spaces of parabolic Higgs bundles use parabolic levels
to define stability. In the quasi-parabolic setting we don't have those parameters.
At the present stage, we therefore include an hypothesis over $\lambda =0$
that guarantees stability for any choice of levels.

\begin{hypothesis}
\label{hypFQPLLC}
For the construction of our moduli
problem, we consider objects $(\lambda , F, \nabla , V)\in QPL(X,\lambda ,r)$
that have only scalar endomorphisms in
the category, and furthermore assume
when $\lambda = 0$ that the Higgs bundle has an irreducible spectral curve.
\end{hypothesis}

With these assumptions as in \cite[Hypothesis 5]{tgps}, 
the moduli problem is unobstructed and framed objects have no endomorphisms.

\begin{theorem}
\label{modulispace}
There is a smooth and separated moduli space,
locally of finite type \cite{Herrero}
$$
\widetilde{M}_{\rm Hod}(X) \stackrel{\lambda}{\longrightarrow} \aaaa ^1.
$$
parametrizing objects 
satisfying Hypothesis \ref{hypFQPLLC}.
\end{theorem}
\begin{proof}
Analogous to the rank $2$ case in \cite[Theorem 6, Proposition 9, Corollary 18]{tgps}.
Numerous references were indicated there. 
\end{proof}

\subsection{Parabolic structures}

A parabolic structure is like a quasi-parabolic structure, except that we don't require the
graded pieces of the filtration to have dimension $1$, and we assign real levels in some interval
of length $1$ to 
the filtration pieces. A useful alternate definition goes as follows: a parabolic bundle consists of
a bundle $F^X$ over $X$, for which we denote $j_{\ast}(F^X)$ by $F(\ast D)$ on $\overline{X}$,
and an increasing 
collection of locally free subsheaves denoted $F_a\subset F(\ast D)$ indexed by 
$a=(a_1,\ldots , a_k)$ for $a_i\in \rr$, satisfying
$F_{a_1, \ldots , a_i + 1, \ldots , a_k} = F_{a_1, \ldots , a_i , \ldots , a_k} \otimes 
\Oo _{\overline{X}}(t_i)$, and
semicontinuous in the sense that 
$F_{\ldots , a_i + \epsilon , \ldots } = F_{\ldots , a_i, \ldots}$ for small $\epsilon$. 
We'll usually denote this kind of structure by $F_{\cdot}$. 

Define the graded pieces at each point $t_i\in D$ to be
$$
{\rm gr}_{t_i,a _i} (F_{\cdot}):= F_{\ldots , a_i, \ldots}/ F_{\ldots , a _i - \epsilon , \ldots }
$$
The jumps or parabolic levels at $t_i$ are the $a_i$ such that this is nonzero. 
For any $k\in \zz$, multiplication by $(z-t_i)^{-k}$ induces 
isomorphisms ${\rm gr}_{t_i,a _i} (F_{\cdot})\cong {\rm gr}_{t_i,a _i+k} (F_{\cdot})$.
One therefore only needs to look at the graded pieces for values of $a_i$ in
an interval of length $1$, usually chosen as $(-1,0]$. 

A logarithmic $\lambda$-connection on $F_{\cdot}$ is a $\lambda$-connection 
operator $\nabla ^{\lambda}$ on $F^X$ that induces a logarithmic $\lambda$-connection
on each subsheaf $F_a$. At each point $t_i\in D$ and for each parabolic level $a_i$
we get an endomorphism 
$$
{\rm res}_{t_i,a_i}(\nabla^{\lambda}) : {\rm gr}_{t_i,a _i} (F_{\cdot})
\rightarrow {\rm gr}_{t_i,a _i} (F_{\cdot}).
$$
Notice that the isomorphism of multiplication by $(z-t_i)^{-k}$  induces
$$
{\rm res}_{t_i,a_i+k}(\nabla^{\lambda}) = {\rm res}_{t_i,a_i}(\nabla^{\lambda}) - k\lambda .
$$

If we fix $b=(b_1,\ldots , b_k)$ then the vector bundle $F_b$ has a 
logarithmic $\lambda$-connection. We would like to give this bundle a 
compatible quasi-parabolic
structure. 

For $t_i\in D$, the fiber 
denoted $(F_b)_{t_i} = F_{t_i,b_i} = F_{\ldots, b_i,\ldots } / F_{\ldots, b_i-1,\ldots }$ has 
a filtration with associated-graded
$$
\bigoplus _{a\in (b_i-1,b_i]} {\rm gr}_{t_i,a}(F_{\cdot}).
$$
The pieces have operators ${\rm res}_{t_i,a}(\nabla )$, whence the decomposition 
into generalized eigenspaces
$$
{\rm gr}_{t_i,a}(F_{\cdot}) \cong \bigoplus _{\alpha \in \cc} {\rm gr}_{t_i,a,\alpha}(F).
$$

\begin{proposition}
\label{pointwise}
Suppose that the nonzero generalized eigenspaces ${\rm gr}_{t_i,a,\alpha}(F)$
have dimension $1$. 
Then for each $t_i\in D$ 
there exists an ordering of the set of $r$ pairs $(a,\alpha )$ where ${\rm gr}_{t_i,a,\alpha}(F)\neq 0$,
written in order 
as $(a^1_i,\alpha ^1_i), \ldots , (a^r_i, \alpha ^r_i)$, satisfying the following condition:
\newline
(O) \, if $\alpha ^j_i = \alpha ^k_i$ then $a^j_i < a^k_i$.
\newline
Use the ordering to define a complete filtration of the fiber $F_{t_i,b_i}$. 
Making these
choices at each singular point yields a framed object of $QPL(X,\lambda ,r)$,
which if it satisfies Hypothesis \ref{hypFQPLLC} gives a point of $\widetilde{M}_{Hod}(X)_{\lambda}$. 
\end{proposition}

The orderings may not be unique, 
and different choices give different points in the moduli space. 
This is the motivation for defining the Hecke gauge groupoid $\srG _{Hod}$
in Section \ref{HGH}
that will relate the points obtained by different such choices.

\subsection{Harmonic bundles}

Suppose given a solution of the Hitchin equations over $X$, expressed
as a metrized ${\mathcal C}^{\infty}$ bundle with a collection of
operators called a harmonic bundle 
$(E,\partial , \overline{\partial}, \varphi , \overline{\varphi},h )$ of rank $r$.
We'll assume throughout that it is {\em tame}, i.e. the eigenvalues of $\varphi$ 
have poles of order $\leq 1$ at the singularities. 

For any $\lambda \in \aaaa^1$ we get a holomorphic structure
$F^{\lambda , X} = (E,\overline{\partial} + \lambda \overline{\varphi})$ 
over the open set $X$. 
The metric induces a notion of meromorphic
growth of sections, yielding the sheaf $F^{\lambda}(\ast D)$,
quasicoherent over $\overline{X}$, consisting of holomorphic sections of 
$F^{\lambda , X}$ whose norms grow at most meromorphically near $D$. 

By \cite{hbnc} for $\lambda = 0,1$ and by 
\cite{Mochizuki} in general, we obtain a parabolic bundle $F^{\lambda}_{\cdot}$ on 
$\overline{X}$. The sheaf of sections at level $a=(a_1,\ldots , a_k)$ 
is the subsheaf $F^{\lambda}_{a _1,\ldots , a _k}\subset F^{\lambda}(\ast D)$
of holomorphic sections of $ (E,\overline{\partial} + \lambda \overline{\varphi})$ 
whose growth, as measured by $h$, is
less than $|z-t_i| ^{-a _i -\epsilon}$ near $t_i$ for any $\epsilon > 0$. 
This is a parabolic bundle. 

The operator $\nabla ^{\lambda} = \lambda \partial + \varphi $
defines a logarithmic $\lambda$-connection 
$$
\nabla ^{\lambda} : F^{\lambda}_{a_1,\ldots , a_k } 
\rightarrow 
F^{\lambda}_{a_1,\ldots , a_k } \otimes \Omega ^1_{\overline{X}}(\log D).
$$
If $\lambda \neq 0$, the rescaled connection operator $\lambda ^{-1} \nabla ^{\lambda}$ gives a flat
${\mathcal C} ^{\infty}$ connection 
$$
D^{\lambda} = \partial + \overline{\partial} + \lambda ^{-1} \varphi + \lambda \overline{\varphi}.
$$
These rescaled connections
appear in the works on conformal limits
\cite{CollierWentworth,MulaseEtAl,DumasNeitzke,Schulz}. 

\begin{remark}
\label{conj}
On the complex conjugate curve $X^c$(denoted that way since we already
use the notation $\overline{X}$ for the compactification) there is a conjugate harmonic
bundle, and the sheaf of flat sections of 
the flat connection corresponding to $\lambda ^{-1}$ on $X^c$, is naturally isomorphic
to the sheaf of flat sections of $D^{\lambda}$ on $X$ via the 
complex conjugation homeomorphism $X^{\rm top}\cong (X^c)^{\rm top}$. 
The local filtrations at nearby points
to $t_i$ by order of growth of flat sections are the same under this isomorphism. 
\end{remark}

\subsection{KMS spectrum}

Following \cite{Mochizuki}, the {\em KMS space} is defined to be 
$$
{\bf KMS}:= (\rr / \zz) \times \cc  = (\rr \times \cc ) / \zz .
$$
We are measuring it at $\lambda = 0$.  Define
the {\em KMS spectrum} of the harmonic bundle 
at a point $t_i\in D$ to be the collection of $r$ pairs $(a_i\alpha _i) \in {\bf KMS}$ 
where $\alpha _i$ are the eigenvalues of the residue of the Higgs field $\varphi$ on 
${\rm gr}_{t_i,a _i}(F^0_{\cdot})$
for $-1 < a_i \leq 0$. The parabolic levels $ a _i$ are viewed as being in $\rr / \zz$. 

For $\lambda \in \aaaa^1$, Sabbah and Mochizuki \cite{Mochizuki} define a map
$({\mathfrak p}_{\lambda} , {\mathfrak e}_{\lambda}) :
\rr \times \cc \rightarrow \rr \times \cc$
by  
\begin{equation}
\label{frakep}
{\mathfrak p}_{\lambda}(a,\alpha ) := 
a + 2 {\rm Re}(\lambda \overline{\alpha})
\;\;\;\;\;\;
{\mathfrak e}_{\lambda}(a,\alpha ) := 
\alpha - a\lambda + \overline{\alpha} \lambda ^2.
\end{equation}
These map $\zz \cdot (1,0)$ to $\zz \cdot (1,-\lambda )$, 
and define an isomorphism 
$$
{\bf KMS} \stackrel{\cong}{\longrightarrow} 
{\bf KMS}_{\lambda} := (\rr \times \cc ) / \zz \cdot (1,-\lambda ).
$$

For $\lambda\in \aaaa^1$, a harmonic bundle yields a 
parabolic logarithmic $\lambda$-connection as explained above. 
Define the KMS spectrum elements $(a ^{\lambda}_i, 
\alpha ^{\lambda}_i)\in {\bf KMS}_{\lambda}$ where $\alpha _i ^{\lambda}$
are the eigenvalues of ${\rm res}_{t_i,a_i}(\nabla ^{\lambda})$.
These are viewed naturally as elements of the quotient
${\bf KMS}_{\lambda}$. 
The Higgs case corresponds
to $\lambda = 0$. 

There is a one-to-one correspondence \cite{Mochizuki} 
between the KMS spectrum elements at $\lambda = 0$
and those at other values of $\lambda$, via the 
isomorphism $({\mathfrak p}_{\lambda} , {\mathfrak e}_{\lambda}) $.

\begin{hypothesis}
\label{hypHB}
Let  $(E,\partial , \overline{\partial}, \varphi , \overline{\partial},h )$ be a tame
harmonic bundle of rank $r$. We assume that 
the spectral curve of the Higgs field is irreducible, and 
that at each $t\in D$, there are  $r$ distinct spectrum elements of ${\bf KMS}$. 
\end{hypothesis}

Given a harmonic bundle, for each $\lambda$ we get a parabolic logarithmic 
$\lambda$-connection. Even under Hypothesis \ref{hypHB}, there will be some
values of $\lambda$ for which the parabolic filtration isn't complete, and hence 
doesn't define a quasi-parabolic structure. When we cross over these points,
the order in the filtration changes. 

Nevertheless, as shall be seen in subsection \ref{smfilt} below, one
can define local sections of the moduli space. This involves using
an order different from the parabolic orderings of Proposition \ref{pointwise}. 
To get the local sections to glue
together in a globally defined section over $\aaaa^1$, we need to define
an equivalence relation on $\widetilde{M}_{\rm Hod}(X)$ so that the images 
of local sections modulo the equivalence
relation agree on overlapping open neighborhoods. 
The equivalence relation should be etale in order to get a smooth
quotient space. We'll see how to do that next.

\section{The Hecke gauge groupoids}
\label{HGH}

Consider first the notion of a collection of partially defined maps from 
a smooth complex manifold $Y$ to itself.
A {\em partial automorphism of $Y$} is a pair $(U,f)$ where $U\subset Y$ is a dense Zariski open 
subset and $f: U \rightarrow Y$ is an isomorphism onto its image $f(U)$ which is also
assumed to be a dense Zariski open subset.  Even though more general partially defined maps are
not necessarily composable (indeed, the image of
a first map could go into the closed subset where a second map is not defined),  the partial  
automorphisms according to this definition are composable,  and have partial inverses.  The 
partial inverse of $(U,f)$ is by definition $(f(U), f^{-1})$ and these compose to the partial identities
$(U,Id)$ or $(f(U),Id)$. 

A collection of partial automorphisms $\{ g_i \} _{i\in I}$ generates a monoid
$\widehat{G}$ 
of partial automorphisms, and hence defines a groupoid 
$(Y, \srG )$ in the
category of complex manifolds.  The object-object is $Y$ and the morphism-object is
a complex manifold $\srG$
with structural maps 
$$
Y\stackrel{i}{ \rightarrow} \srG \stackrel{s,t}{\longrightarrow} Y 
$$
with multiplication $\mu : \srG \times _Y \srG \rightarrow \srG$ and inverse $\nu : \srG \rightarrow 
\srG$.  This is constructed in the following way: let 
$$
\widehat{\srG} := \coprod _{(U,g)\in \widehat{G}} U
$$
that already has a structure of groupoid.  The points of $\widehat{\srG}$ are triples
of the form $(x,U,g)$ where $(U,g) \in \widehat{G}$ and $x\in U$. 
The source and target are
$$
s(x,U,g) := x, \;\;\; t(x,U,g) := g(x).
$$
Then,  define an equivalence relation $\sim$ on 
$\widehat{\srG}$ by saying that 
$$
(x,U,f) \sim (y,V, g)
$$
if and only if $x = y$ and $f(x) = g(y)$ and the map of germs $f: (Y,x) \rightarrow (Y,f(x))$
is equal to the map of germs $g: (Y,y) \rightarrow (Y,g(y))$.  Let $\srG$ be the quotient
of $\widehat{\srG}$ by $\sim$.  
Composition of partial automorphisms defines the composition on $\widehat{\srG}$
and this is compatible with $\sim$ so it descends to a composition on $\srG$. 

\subsection{The groupoid on the Betti moduli space}

Let's first look at how this applies to the Betti moduli space for filtered local systems,
essentially corresponding to one of the Fock-Goncharov spaces \cite{FockGoncharov}.
This is the parameter space of $(\rho , \{ V_{\cdot}^t\} _{t\in D} )$ where
$\rho : \pi _1(X,x)\rightarrow GL_r(\cc)$ is a representation 
and for each $t\in D$, $\{ V_{i}^t\subset \cc^r\} _{0\leq i \leq r}$ is a complete filtration 
invariant under $\rho (\gamma _t)$ for the paths $\gamma _t$ going around
singular points $t\in D$.  

\begin{hypothesis}
\label{hypFLS}
Make the condition that $(\rho , V_{\cdot})$ has only
scalar endomorphisms. 
\end{hypothesis}

Let $M_B$ be the moduli space of $(\rho , V_{\cdot})$ satisfying this condition. Apply
the general theory of the previous subsection to $Y=M_B$. 

For each $t\in D$ we obtain the ordered collection 
$(\alpha _1(t),\ldots , \alpha _r(t))$ of eigenvalues of $\rho (\gamma _t)$ acting on
the $1$-dimensional graded pieces $V_{i}^t/V_{i-1}^t$.
Putting these together for $t_1,\ldots , t_k$ we obtain
a map 
$$
A: M_B \rightarrow \aaaa ^{rk}.
$$

Suppose $\sigma = \{ \sigma (t)\} _{t\in D} \in S_r^k$ is a collection of permutations. 
We define the open subset $U_{\sigma} \subset  \aaaa ^{rk}$ to be the
subset of points 
$$
\alpha = \{ \alpha ^i(t)\} _{t\in D, i=1,\ldots ,r}
$$
such that if $i<j$ and  $\sigma (t)(i) > \sigma (t)(j)$ then $\alpha ^i(t) \neq \alpha ^j(t)$. 
Thus $\alpha \in U_{\sigma}$ if and only if $\sigma$ preserves  the ordering of indices that
have the same values $\alpha ^i(t)$.  

Let 
$$
p_{\sigma} : U_{\sigma} \rightarrow  \aaaa ^{rk}
$$
be the map obtained by applying $\sigma$ to the values.  The image $p_{\sigma}(U_{\sigma})$
is equal to $U_{\sigma ^{-1}}$ and $p_{\sigma ^{-1}}$ is inverse to $p_{\sigma}$
on these open subsets. 

For any $\sigma$ we obtain a map, defined by reordering the basis elements adapted
to the filtrations
$$
P_{\sigma} : A^{-1}(U_{\sigma}) \rightarrow A^{-1}(p_{\sigma}(U_{\sigma}))\subset M_B.
$$
This gives a partial automorphism $(U_{\sigma}, P_{\sigma})$ of $M_B$. 
The following are left to the reader. 

\begin{lemma}
Given multi-permutations $\sigma$ and $\tau$,  the composition 
$$
(U_{\tau} , P_{\tau}) \circ (U_{\sigma} , P_{\sigma}) = (U', P')
$$
is a subset of $(U_{\tau \sigma}, P_{\tau \sigma})$ in that $U'\subset U_{\tau \sigma}$ and
$P' = P_{\tau \sigma}|_{U'}$.
\end{lemma}

\begin{corollary}
The groupoid $\srG_B$ on $M_B$ 
defined by this collection of partial automorphisms is equal to the 
union of the graphs of the $(U_{\sigma}, P_{\sigma})$. 
\end{corollary}

In particular, the diagonal is a connected component of $\srG_B$. 

\begin{corollary}
\label{diagonal}
The map $\srG_B\rightarrow M_B\times M_B$ is injective. Any connected open subset of $\srG_B$
that intersects the diagonal is contained in the diagonal. 
\end{corollary}

\subsection{The groupoid on the Hodge moduli space}

The general construction applies also to define the {\em Hecke gauge groupoid} acting on 
the Hodge moduli space.  We need to 
specify the generating collection of partial automorphisms. 

Let $Y=\widetilde{M}_{\rm Hod}$
denote the moduli space of Theorem \ref{modulispace}.
A point of $Y$ corresponds to 
$(\lambda , F, \nabla, V, \beta )$ (usually denoted just $F$) 
as described in subsection \ref{qp-struct}.

We have maps $\lambda : Y\rightarrow \aaaa^1$ and for each $t\in D$ and $k\in 
\{ 1,\ldots , r\}$ maps 
$$
\theta _{t,k} : Y \rightarrow \aaaa^1
$$
where $\theta _{t,k}(F)$ is the eigenvalue of the residue of $\nabla$ at $t$, in the $k$-th piece of the associated-graded of the filtration.

The first automorphisms are $H_t$, whose domain of definition is $Y$, corresponding to a single Hecke rotation 
\cite{Maruyama,InabaIwasakiSaito,FassarellaLoray,HuHuangZong,Matsumoto}
at the point $t\in T$. The locally free sheaf $F'$ underlying $H_t(F)$ 
is the kernel of $F\rightarrow i_{t,\ast}(F_t/ V^t_{r-1})$ where $i_t:\{ t\} \rightarrow
\overline{X}$ is the inclusion. The filtration of $(F')_t$ is defined by
setting $V'_i$ to be the kernel of $(F')_t \rightarrow (F_t / V^t_{i-1})$. We have
$$
\theta _{t,k}(H_tF) = \theta _{t,k-1}(F), \;\; \mbox{ for } 2\leq k \leq r
$$
and 
$$
\theta _{t,1}(H_tF) = \theta _{t,r}(F) + \lambda .
$$

The other partially defined automorphism is denoted $T_{t,i}$ for $t\in D$ and $1\leq i < r$. 
This transposes
the $i$ and $i+1$ pieces of the residue. Its domain of definition is the subset of points 
of $Y$ where $\theta _{t,i}(F) \neq \theta _{t,i+1}(F)$. The underlying bundle
of $T_{t,i}(F)$ is again $F$, but the new filtration is obtained by taking a basis
adapted to the original filtration and composed of generalized eigenvectors of 
${\rm res}_t(\nabla )$, and transposing the $i$ and $i+1$ vectors. 
We have
$$
\theta _{t,k}(T_{t,i}F) = \theta _{t,k}(F), \;\; \mbox{ for } k \neq i,i+1
$$
and 
$$
\theta _{t,i}(T_{t,i}F) = \theta _{t,i+1}(F), \;\;\;\; 
\theta _{t,i+1}(T_{t,i}F) = \theta _{t,i}(F).
$$

We also have a simple tensoring operation $U_t$ for $t\in D$ that sends 
$F$ to $F\otimes \Oo _{\overline{X}}(t)$ with filtrations corresponding in an easy way.
For these, $\theta _{t,i}(U_tF) = \theta _{t,i}(F) - \lambda$.

We note that $T_{t,i}^2$ is the identity on the open subset of definition, and 
$U_tH_t^r = H_t^r U_t$ is the identity on $Y$. 

Let $\widehat{G}$ be the monoid of partial automorphisms 
generated by these operations. It has an action by partially
defined transformations on $Y$.  Let $\srG_{Hod}$ 
be the etale groupoid on the Hodge moduli space
$\widetilde{M}_{\rm Hod}(X)$
generated by the above operations.

\begin{proposition}
\label{injectivity}
Suppose $F\in Y= \widetilde{M}_{\rm Hod}$ over $\lambda \neq 0$, and suppose $A,B\in \widehat{G}$ are two words.  Suppose that
$A(F)$ and $B(F)$ are defined and equal. Then there is an open neighborhood 
$F\in U\subset Y$ such that $A$ and $B$ are defined and coincide on $U$.
\end{proposition}
This will be proven in the next subsection using the Riemann-Hilbert correspondence. 

\begin{corollary}
The map
$$
(s,t):\srG _{Hod} \rightarrow 
\widetilde{M}_{\rm Hod}(X) \times _{\aaaa^1}
\widetilde{M}_{\rm Hod}(X)
$$
is injective over $\lambda \in \Gm \subset \aaaa^1$. 
\end{corollary}

\subsection{The Riemann-Hilbert correspondence}

The Riemann-Hilbert correspondence provides an identification between moduli
of connections and moduli of representations, and has been studied in the
quasi-projective case \cite{Deligne,EsnaultViehweg,AlfayaGomez,
Boalch,BudurLererWang,HaiEtAl,InabaIwasakiSaito,MisraSingh,NitsureSabbah,Sage,Singh}.

Let $\widetilde{M}_{dR}(X)$ be the fiber of 
$\widetilde{M}_{Hod}(X)$ over $\lambda = 1$. Rescaling the
$\lambda$-connection to a connection provides an isomorphism
$$
\widetilde{M}_{Hod}(X) \times _{\aaaa^1} \Gm \cong 
\widetilde{M}_{dR}(X)\times \Gm .
$$
We may therefore
consider, for the moment, the de Rham fiber.
Let $\srG _{dR}$ be the restriction of $\srG _{Hod}$ to $\widetilde{M}_{dR}(X)$.

We would like to define the Riemann-Hilbert map
$$
\widetilde{rh}:\widetilde{M}_{dR}(X)  \rightarrow M_B.
$$
It is going to be defined locally over open subsets, with the different local pieces
glueing together to provide a map that is well defined modulo the Betti groupoid
$\srG_B$. For this, we'll put back real parabolic levels into the picture. This procedure is
motivated by the construction used in subsection \ref{smfilt} below due to Sabbah and
Mochizuki. 

Recall that 
$\theta : \widetilde{M}_{dR}(X) \rightarrow \aaaa ^{kr}$ is the map with coordinate functions
$\theta _{t,j}$ defining the eigenvalues of the residue of the connection
on the graded pieces of the filtration. Suppose $B\subset  \aaaa ^{kr}$
is a small ball centered around some point. Choose real numbers $b_{t,j}\in (-1,0]$
with $b_{t,j-1} < b_{t,j}$ for $j=1,\ldots , r$ at each $t\in D$, and use these
to put a parabolic structure $F_{\cdot}$ on each quasi-parabolic connection $(F,\nabla , V_{\cdot})$. 
Precisely, 
$$
F_{a_1,\ldots , a_k} := \ker [
F\rightarrow \bigoplus _{a_i < b_{t_i,j}}i_{t_i,\ast} (F_{t_i} / V_{t_i,j}) ] .
$$
This parabolic structure defines a filtered local system $(L,V_{\cdot ,B})$ \cite{hbnc}. 
The real jumps of the filtration are $b_{t_i,j} + {\rm Re}\, \theta _{t_i,j}(F,\nabla , V_{\cdot})$. 

\begin{remark}
\label{choiceb}
If $B$ is a small enough ball around any point of $\aaaa^{kr}$, then there 
exists a choice of $b_{t,j}\in (-1,0]$ such that
for any $\eta \in B \subset  \aaaa ^{kr}$, and at any $t_i\in D$, the real numbers
$b_{t_i,j} + {\rm Re} \, \eta _{t_i,j}$ are distinct. 
\end{remark}

For $(F,\nabla , V_{\cdot}, \beta ) \in \theta ^{-1}(B)$, and for $b$ chosen as in the 
Remark \ref{choiceb}, the filtered local system associated to the parabolic
bundle defined from $F$ using $b$, has graded pieces of the filtrations of dimension $1$.
This yields a well-defined Betti quasi-parabolic structure and hence
a point in $M_B$. Our hypotheses \ref{hypFQPLLC} and \ref{hypFLS}
correspond under any map obtained from a choice of levels $b$. 

\begin{proposition}
\label{rhc}
We get in this way a holomorphic map $\theta ^{-1}(B)\rightarrow M_B$.
If we make a different choice of $b$ also valid for the open set $B$, then the
two maps differ by a section $\theta ^{-1}(B)\rightarrow \srG _B$ to the Betti gauge groupoid. 
This gives a well-defined map from
$\widetilde{M}_{dR}(X)$ to $(M_B , \srG _B)$
sending points differing by $\srG _{dR}$ to points differing by $\srG_B$,
and descending to an equivalence
$$
(\widetilde{M}_{dR}(X),\srG_{dR})  \stackrel{\cong}{\longrightarrow} (M_B , \srG _B).
$$
\end{proposition}
\begin{proof}
Consider a point and a choice of $b$ for which some multiplicities of the graded pieces of the
local system are $>1$. There are distinct $j\neq j'$ such that
$$
b_{t_i,j} + {\rm Re}\, \theta _{t_i,j}(F,\nabla , V_{\cdot})
=
b_{t_i,j'} + {\rm Re}\, \theta _{t_i,j'}(F,\nabla , V_{\cdot}).
$$
Now  $0 < | b_{t_i,j} - b_{t_i,j'} | < 1$ so 
$ \theta _{t_i,j}(F,\nabla , V_{\cdot})$
and $\theta _{t_i,j'}(F,\nabla , V_{\cdot})$ 
are different 
and their real parts differ by $<1$. It follows that their
exponentials are different, so the two filtration levels can be interchanged 
via $\srG _B$, corresponding to $b$ crossing the corresponding wall. This shows that
the map to $ (M_B , \srG _B)$ is well-defined. We can calculate that the generators of 
$\srG _{dR}$ go to elements of $\srG_B$. To show the last statement that we get an
equivalence of groupoids, we need to show that two points whose images in 
$M_B$ differ by $\srG_B$, were in fact different by $\srG_{dR}$. For this, using elements
of $\srG_{dR}$ we can modify a quasi-parabolic $\lambda$-connection into one satisfying Deligne's
condition \cite{Deligne} that the real parts of the $ \theta _{t_i,j}(F,\nabla , V_{\cdot})$
are all contained in $(-1,0]$. There is $\epsilon$
so that they  are in the open interval 
$(-1+\epsilon , \epsilon)$. On those points, the relations given by $\srG_B$
lift to $\srG_{dR}$. 
\end{proof}

\begin{corollary}
\label{rhcHod}
Let
$\widetilde{M}_{Hod}(X)_{\Gm}:= \widetilde{M}_{Hod}(X) \times _{\aaaa^1} \Gm
\cong \widetilde{M}_{dR}\times \Gm$, and let $\srG_{Hod,\Gm}$ be the
restriction of $\srG_{Hod}$ over $\Gm$. We get
a {\em Riemann-Hilbert} equivalence of groupoids
$$
\widetilde{M}_{Hod,\Gm}(X),\srG_{Hod,\Gm}  \stackrel{\cong}{\longrightarrow}
(M_B , \srG _B)\times \Gm .
$$
\end{corollary}

\begin{corollary}
\label{bettipoints}
The set of points of the quotient space $\widetilde{M}_{\rm dR}(X) /  \srG_{\rm dR}$ 
is isomorphic to the set of points of $\widetilde{M}_B(X)  / \srG _B$, and this has the
following description: it is the set of isomorphism classes of framed local systems 
provided with a complete filtration of each generalized eigenspace
of the local monodromy operators, such that the filtrations are compatible
with the local monodromy operators, and such that 
the resulting structure satisfies the condition \ref{hypFLS} (which only depends
on the filtrations of the generalized eigenspaces). The same holds for points of 
$\widetilde{M}_{\rm Hod}(X) / \srG_{\rm Hod}$ over any $\lambda \neq 0$. 
\end{corollary}

\begin{proof}[Proof of Proposition \ref{injectivity}]
In the situation of Proposition \ref{injectivity}, compose with further words to get
into the region $R\subset \widetilde{M}_{\rm dR}(X)$ where 
the real parts of the $ \theta _{t_i,j}(F,\nabla , V_{\cdot})$
are all contained in an open interval 
$(-1+\epsilon , \epsilon)$. We have an open subset of $\srG_{dR}$ containing a point
on the diagonal. It maps to an open subset of $\srG_B$ containing a point of the
diagonal, so by Corollary \ref{diagonal} it is a subset of the diagonal in $\srG_B$. 
Now, since the map from $R$ to $M_B$ is injective, it follows that all the points
in the open subset are in the diagonal of $R$, showing the desired property
for Proposition \ref{injectivity}. 
\end{proof}

Let $X^c$ denote the complex conjugate curve defined by conjugating
all the coefficients of the equations defining $X$. Then 
there is a natural equivalence between the Betti groupoids for $X$ and $X^c$.
The Riemann-Hilbert equivalence of Corollary \ref{rhcHod} gives 
$$
(M_{Hod}(X)_{\Gm}, \srG _{Hod,\Gm}) \cong (M_{Hod}(X^c)_{\Gm}, \srG _{Hod,\Gm})
$$
over $\lambda \mapsto \lambda ^{-1}$ on $\Gm$. 
This identification can be used to define the Deligne-Hitchin groupoid
$ ( \widetilde{M}_{\rm DH}, \, \srG _{\rm DH} )$ as in \cite{AlfayaGomez,BeckHellerRoser,tgps}.

\section{Preferred sections of the Deligne-Hitchin moduli space}

Suppose given a harmonic bundle. For each $\lambda \in \aaaa^1$ and $b=(b_1,\ldots , b_k)\in \rr^k$
there is a vector bundle with 
logarithmic $\lambda$-connection $(F^{\lambda}_b,\nabla ^{\lambda})$.
These also have parabolic structures. 
Apply the procedure of Proposition \ref{pointwise}
to get a point of $\widetilde{M}_{Hod}(X)_{\lambda}$. 

By the technique to be described below, 
the point is well-defined modulo $\srG_{Hod}$, and this glues with the corresponding
construction on $X^c$ to define a set-theoretical section 
$\rho : \pp^1 \rightarrow ( \widetilde{M}_{\rm DH}, \, \srG _{\rm DH} )$.

The construction 
of Proposition \ref{pointwise} doesn't work well as $\lambda$ varies. 
Indeed, as we have seen, the parabolic levels where the filtrations jump
vary as a function of $\lambda$. So, for a fixed $b$, the bundles $F^{\lambda}_b$
will jump as ${\mathfrak p}_{\lambda}$ crosses $b$ modulo $\zz$. We don't immediately get a 
holomorphic section
$\aaaa^1\rightarrow \widetilde{M}_{\rm Hod}(X)$ in this way, and not even
a locally defined one. This difficulty was pointed out and resolved in 
\cite{Mochizuki,Sabbah}.

\subsection{Sabbah-Mochizuki filtrations}
\label{smfilt}

The problem of holomorphic variation of the filtration may be remedied by
considering only values of $b$ that don't cross the ${\mathfrak p}_{\lambda}$ for 
$\lambda \in U$ in a small neighborhood. 
This is due to one of the main constructions by
Sabbah and Mochizuki, see \cite{Mochizuki,Sabbah}. Work here near a fixed
singular point $t\in D$. 

More precisely, fix $\lambda _0\in \aaaa^1$ and consider a value $b_0$ that is a parabolic
level at $\lambda_0$. Then we may choose $b=b_0+\eta$ and $b' = b_0-\eta$ (with $2\eta < 1$)
and 
a small open neighborhood $\lambda _0\in U \subset \aaaa^1$, such that $b_0$ is the only
parabolic level between $b'$ and $b$ at $\lambda _0$, and such that none of the 
${\mathfrak p}_{\lambda}$ are equal to $b$ or $b'$ for any $\lambda \in U$. 
Then Mochizuki shows \cite{Mochizuki} 
that there is a bundle, that we could denote by $F^U_b$, over
$U\times \overline{X}$, whose restrictions to $\{ \lambda \} \times \overline{X}$ are
$F^{\lambda}_b$ for any $\lambda \in U$. It has subsheaf 
$F^U_{b'}\subset F^U_b$ such that 
${\rm Gr}^U_{(b',b)} := F^U_{b}/F^U_{b'}$ is a vector bundle over $U \times \{ y\}$. 
Furthermore, $F^{\lambda}_b$ has an operator $\nabla ^U$ that is a relative $\lambda$-connection,
here $\lambda$ denoting the coordinate in $U$, preserving $F^U_{b'}$, such that the
resulting operator ${\rm res}_{(b',b), y} (\nabla ^U)$ is a holomorphic endomorphism
of the bundle ${\rm Gr}^U_{(b',b)}$. 

At $\lambda = \lambda _0$, the fiber of ${\rm Gr}^U_{(b',b)}$ is just ${\rm gr}_{t,b_0}F^{\lambda_0}_{\cdot}$. By Hypothesis \ref{hypHB}, the (generalized) 
eigenspaces of 
${\rm res}_{t,(b',b),} (\nabla ^U)$ at $\lambda = \lambda _0$ have rank $1$. Therefore, possibly
shrinking $U$ we may assume that this is also the case for all $\lambda \in U$. This
gives a decomposition of ${\rm Gr}^U_{(b',b)}$ into a direct sum of line bundles that we'll
call the {\em eigen-bundles}. The eigenvalues of ${\rm res}_{t,(b',b)} (\nabla ^U)$ on these
bundles vary according to the formula \eqref{frakep} for ${\mathfrak e}_{\lambda}$. 

Choosing an ordering of the set of eigen-bundles yields a refinement of the filtration by $F^U_{b'}\subset
F^U_b$ into a filtration whose quotients are of rank $1$ over $t$. 

Doing this for each parabolic level $b_0$ in an interval of length $1$, we obtain a
quasi-parabolic structure and hence a holomorphic map $\rho ^U: U\rightarrow \widetilde{M}_{\rm Hod}(X)$.

\begin{lemma}
If $\lambda \in U$ then the image of $\lambda$ under this map differs from any
of the points constructed in Proposition \ref{pointwise} at $\lambda$ by an element of
the Hecke gauge groupoid $\srG _{Hod}$. 
\end{lemma}

\begin{corollary}
\label{cor-defsec}
We may choose the neighborhoods in such a way that for any two
with nonempty intersection, there is a unique holomorphic
map $U\cap U' \rightarrow \srG _{Hod}$ such that the composition
$$
U\cap U' \rightarrow \srG _{Hod} \hookrightarrow 
\widetilde{M}_{\rm Hod}(X) \times _{\aaaa^1}
\widetilde{M}_{\rm Hod}(X)
$$
coincides with the map $(\rho ^U |_{U\cap U'}, \rho ^{U'} |_{U\cap U'})$. 
Thus, the collection of sections $\rho ^U$ provides a well-defined map to the
groupoid
$$
\rho : \aaaa^1 \rightarrow  ( \widetilde{M}_{Hod}(X),  \srG _{Hod}).
$$
\end{corollary}
\begin{proof}
Let $\Delta \subset \Gm$ be the set of points where two different KMS-spectrum
elements have the same eigenvalues, modulo $\lambda \zz$, at some point of $D$. 
It is discrete in $\Gm$, since it is a finite union of inverse images of $\zz$ by 
functions of the form $\lambda \mapsto a\lambda ^{-1} + b + c\lambda$ that
are not $\zz$-valued constants due to Hypothesis \ref{hypHB}. 
Choose a neighborhood $U_0$
at $\lambda = 0$, then small neighborhoods at the 
points of $\Delta$ that are not in $U_0$; then cover the remainder of $\aaaa^1$ by neighborhoods that
don't intersect $\Delta \cup \{ 0\}$. With this convention, for any two different
neighborhoods $U$ and $U'$, the intersection $U\cap U'$ is contained in $\Gm - \Delta$,
in other words the eigenvalues of the residues are distinct over $U\cap  U'$. 
Over this intersection,
two holomorphic sections of $\widetilde{M}_{Hod}(X)$ 
that differ pointwise by elements of $\srG _{Hod}$ 
differ by a holomorphic function $U\cap U' \rightarrow \srG _{Hod}$.
\end{proof}

\subsection{Preferred sections}

\begin{theorem}
\label{twistorsections}
Given a framed tame harmonic bundle, and assuming Hypothesis \ref{hypHB}, 
the sections defined in Corollary \ref{cor-defsec} for $X$ and $X^c$ glue
to a holomorphic {\em preferred section} of the Deligne-Hitchin groupoid
$$
\rho : \pp^1\rightarrow (  \widetilde{M}_{\rm DH}(X), \, \srG_{\rm DH} )  .
$$
\end{theorem}
\begin{proof}
Over the chart $\aaaa^1\subset \pp^1$, the section is defined using Corollary \ref{cor-defsec}. 
On the other chart, we do the same construction using the complex conjugate
harmonic bundle on $\overline{X}$ to get a section defined on the other chart 
$\aaaa^1\subset \pp^1$,
the neighborhood of $\infty$. These glue together to define a section into 
$ (\widetilde{M}_{\rm DH}(X), \, \srG_{\rm DH} )$. Pointwise agreement 
may be seen by using the
description of points in Corollary \ref{bettipoints}, together with the 
identification between the filtrations of sheaves of flat sections by order of growth of
Remark \ref{conj}. Then argue as in \ref{cor-defsec} to obtain the glueing. 
\end{proof}

\section{The neighborhood of a preferred section}

Our next goal is to understand the structure of the neighborhood of a preferred section. 
We state here the result for the normal or relative tangent bundle. 
Let $T( (\widetilde{M}_{\rm DH}(X), \srG_{\rm DH})/\pp^1)$ 
denote the relative tangent bundle of the 
Deligne-Hitchin groupoid over $\pp^1$. 

\begin{theorem}
\label{twistorproperty}
Let $\rho$ be a preferred section associated to a tame harmonic bundle by 
Theorem \ref{twistorsections}. The pullback 
$T_{\rho} := \rho ^{\ast} T((\widetilde{M}_{\rm DH}(X), \srG_{\rm DH})/\pp^1)$ 
is a mixed twistor
structure \cite{twistor} with weight filtration characterized by: 
\begin{itemize}

\item 
$W_{-1}=0$

\item
$W_0$ is the trivialized tangent space of the change of framing

\item
$W_1$  is the subspace of deformations that fix the eigenvalues of the residues 

\item
$W_2$ is the full space. 

\end{itemize}
\end{theorem}
\begin{proof}
One may do a proof similar to the proof for rank $2$,
based on the pure twistor ${\mathcal D}$-module theory of \cite{Mochizuki,Sabbah}, 
as described in \cite{tgps}. 
See also the remarks there on an alternative strategy that was sketched by Mochizuki using a
mixed twistor ${\mathcal D}$-module 
\cite{MochizukiMTM} denoted ${\mathfrak T}[\ast D][!x]$. 
\end{proof}

The mixed twistor property means that 
$W_k/W_{k-1}\cong \Oo _{\pp^1}(k)^{\oplus n_k}$ for $k=0,1,2$.
The weight $1$ subquotient 
$W_1/W_0$ is the tangent of moduli of connections 
fixing the local conjugacy classes, yielding
a hyperkähler direction \cite{Nakajima,Holdt}. The piece $W_2/W_1$ governs the change of 
KMS spectrum elements, for which the weight $2$ philosophy has been discussed 
in \cite{weight2,tgps}.

\subsection{Construction of a separated neighborhood}

The tangential result of Theorem \ref{twistorproperty} leads naturally to a mixed twistor property
for the full formal neighborhood. The first task is to construct a separated neighborhood
of a preferred section, within the non-separated moduli space groupoid. 

In our notational conventions\footnote{Letters may however be used differently in this subsection
than in the rest of the paper.}, 
a smooth analytic groupoid consists of a separated complex
manifold $Y$ together with a complex manifold also denoted by $\srG$ provided with 
maps 
$$
s,t: \srG \rightarrow Y  \mbox{ (etale)}, \;\;\; 
m: \srG \times _Y \srG \rightarrow Y, \;\;\; e: Y\rightarrow \srG
$$
making a groupoid internal to the category of smooth analytic spaces.

Suppose given a separated smooth complex manifold $Z$ with a smooth
$\srG$-invariant 
map $p : Y \rightarrow Z$ 
and a section $f : Z \rightarrow (Y,\srG )$.  The section is 
given by an open covering $\{ Z_i\subset Z\}$, maps
$f_i: Z_i \rightarrow Y$
and $\rho _{ij} : Z_{ij} \rightarrow \srG$ with $p \circ f_i = Id _{Z_i}$. 
In the application $Z$ will be the twistor $\pp^1$ but for inductive purposes we envision 
$Z$ noncompact provided with a compact subset. 

\begin{proposition}
\label{cpt-nbd}
In the above situation, suppose $K\subset Z$ is a compact subset.  
Then there exists an open neighborhood 
$K\subset Z' \subset Z$ and a separated 
complex manifold $X'$ with a smooth map $X'\rightarrow Z'$ and 
an etale map $e':X'\rightarrow (Y,\srG )$ together with a lifting $f ' : Z' \rightarrow X'$
and an isomorphism $e\circ f' \cong f|_{Z'}$, i.e. a map $Z'\rightarrow \srG$
 source $e\circ f'$ and target $f|_{Z'}$. 
\end{proposition}

We first prove uniqueness. 

\begin{lemma}
\label{uniqueness}
If it exists, $(Z',X',e')$ is unique up to reducing the size of the neighborhoods, 
up to unique isomorphism.  
\end{lemma}
\begin{proof} 
Suppose $(Z'', X'', e'')$ is another such. We may go to open coverings of $Z'$ and $X'$ (cutting the compact subset into a union of compact subsets in the open pieces of the covering of $Z'$) so
we may assume that the maps factor through 
$h: X' \rightarrow Y$ and $h'': X'' \rightarrow Y$. 
We have an etale map 
$$
(e',e''): X' \times X'' \rightarrow Y \times Y
$$
and $( f',f'') : Z' \cap Z'' \rightarrow X' \times X''$,  that lifts into a map
$$
Z'\cap Z'' \rightarrow  \srG \times _{Y\times Y} X' \times X''.
$$
Set $Z_3 := Z'\cap Z''$ and 
$X_3 := \left( \srG \times _{Y\times Y} X' \times X'' \right) \times _Z Z_3$. 
We note that by construction we get $f_3 :Z_3 \rightarrow X_3$ and either of the
projections $X_3 \rightarrow Y$ is a 
local homeomorphism.  The same is therefore true of the projections $X_3\rightarrow X'$ 
and $X_3 \rightarrow X''$.  

By Remark \ref{cpxNbd} below,  we may reduce the size of $X_3$ so that these maps give open
embeddings to $X'$ and $X''$. This provides the reduction of size on which our two maps agree.
The isomorphism of agreement is given by the lifting of $X_3$ into $\srG$.  Such a lifting is unique 
over a possibly smaller neighborhood since we are provided already with the map from $K$
into $\srG$. 
\end{proof}

\begin{remark}
\label{cpxNbd}
Suppose $p:X \rightarrow Y$ is an etale map between
manifolds, and $K\subset X$ is a compact subset. Suppose that $p|_K$ is injective. 
Then there exists an open neighborhood $K\subset U\subset X$ such that the map 
$p|_U : U \rightarrow Y$ is an open embedding.
\end{remark}

\begin{proof}[Proof of Proposition \ref{cpt-nbd}]
Given a solution $X'$ for a compact subset $K'\subset K$ we can add that to the collection
of charts for $Y$. In this way it is possible to proceed by induction, and reduce to the
case where there
are only two open sets: $Z=Z_1\cup Z_2$ and $K = K_1 \cup K_2$ with $K_i$ compact in $Z_i$,
and $f$ is defined by $f_i: Z_i \rightarrow Y_i:= p^{-1}(Z_i)$. 

Then $Y_1$ and $Y_2$ are both solutions to the problem for $K_{12}:= K_1\cap K_2$. By 
Lemma \ref{uniqueness} there is a manifold $W$ with 
open inclusions $\nu _i : W \stackrel{\cong}{\rightarrow} W_i \subset Y_i$,
with a map $W\stackrel{\rho}{\rightarrow} \srG$
relating the two resulting maps $W\rightarrow Y$, and 
a neighborhood $K_{12}\subset Z_W\subset Z$ and 
a section $f_W: Z_W\rightarrow W$ inducing $f_i : Z_W \rightarrow W_i\subset Y_i$. 

We have $\nu _1^{-1}(f_1(K_1)) \cap \nu _2^{-1}(f_2(K_2)) = f_W(K_{12})$. 
This is seen using the map $p:Y\rightarrow Z$, which
is $\srG$-invariant so $p\circ \nu _1 = p\circ \nu _2$ gives a map $W\rightarrow Z$. 
Any intersection point would have to map to a point in $K_{12}$,
and since it is in the image of the sections it would be in $f_W(K_{12})$.
Note that a statement analogous to the proposition isn't true if we don't include
the map to $Z$ in the hypothesis. 

Choose a relatively compact neighborhood 
$f_W(K_{12}) \subset W' \subset W$ and let $S'=\partial W'$ be its boundary, a compact
subset of $W$. Then $\nu _i^{-1}(f_i(K_i)) \cap S'$ 
are disjoint closed subsets of $S'$ because of the previous paragraph. 

It now follows that 
$\nu _i^{-1}(f_i(K_i)) \cap S'$ are separated by disjoint open neighborhoods
whose complements are two closed, hence compact subsets $T_i$ covering $S'$ such that 
$\nu _i^{-1}(f_i(K_i)) \cap T_i = \emptyset$. 

Put $X'_i:= Y_i - \nu _i(T_i)$.
This is an open neighborhood of $f_i(K_i)$ with $W'_i \subset X'_i$. 
Let $X'$ be the pushout of $X'_1 \leftarrow W' \rightarrow X'_2$. 

We claim that $X'$ is separated. Suppose given a sequence of points that we may assume in $W'$, 
having
limit points in $X'_1$ and $X'_2$.
Since $W'_i$ is relatively compact in $W_i$, the sequences limit to a unique
point of $\overline{W}'$ in $W$. Since $Y_i$ are separated and contain copies of $W$,
the limit points in $X'_i$ correspond in $Y_i$ to this point of $\overline{W}'$. If the point
is in the boundary $S'$, then since it is in both $X'_i$, it has to be in the complement 
of each $T_i$, a contradiction. 
Thus, the limit is in $W'$ so it is unique in $X'$, completing the proof that $X'$ is separated.

Now, the separated pushout is a manifold and gives the required space for the proposition. 
\end{proof}

\subsection{Mixed twistor structure on the completion}

In this subsection we complete the proof of the last part of Theorem \ref{maintheorem}, namely
the construction of the mixed twistor structure \cite{twistor} on the formal completion. Once we have a
separated and smooth local neighborhood, and once we know the mixed twistor
structure on the normal bundle, the mixed twistor structure on the formal
completion follows. 

Consider a preferred section $\rho : \pp^1 \rightarrow  
(\widetilde{M}_{DH},\srG_{DH})$ corresponding by Theorem \ref{twistorsections} to
a harmonic bundle satisfying Hypothesis \ref{hypHB}.  By Proposition \ref{cpt-nbd}
there is a separated complex
analytic space with etale map $M \rightarrow (\widetilde{M}_{DH},\srG_{DH})$
and a section projecting to the preferred section, that 
we'll also denote $\rho : \pp^1 \rightarrow M$.
Let ${\mathfrak m}\subset  \Oo _{M }$ 
denote the ideal defining the image of $\rho$ in $M$.  For any $n$
we have a coherent sheaf of $\Oo _{\pp^1}$-algebras over $\pp^1$
$$
A_n := \Oo _{M } / {\mathfrak m}^{n+1},\;\;\;\; A_0 =\Oo _{\pp^1}.
$$
These fit together in an inverse system 
$\widehat{A} := \lim _{\leftarrow} A_n$ to give the formal completion of the
neighborhood along the section $\rho$.  This may be viewed as a pro-object in the category
of sheaves of algebras over $\pp^1$.  We denote also by ${\mathfrak m}\subset \widehat{A}$
the kernel of $\widehat{A}\rightarrow A_0$.

The group $GL_r(\cc )$ acts on $\widetilde{M}_{DH}$
by change of framing. Let $N\subset M$ be the orbit of $\rho$ under 
this action. Locally around $\rho$ it looks like $\pp^1$ times a small neighborhood of the
identity in $GL_r(\cc )$. 
Theorem \ref{twistorproperty} gives a filtration
$$
0 \subset T_{\rho}N/\pp^1 = W_0 \subset W_1 \subset 
W_2 = T_{\rho}M/\pp^1
$$
that extends, via the action of $GL_r(\cc )$ on framings, 
to a filtration 
$$
0 \subset T(N/\pp^1) =W_{0,N} \subset W_{1,N} \subset W_{2,N} = T(M/\pp^1) |_N 
$$
with the $W_{1,N}$ part characterized
as the deformations that fix the eigenvalues of the residules. 

Let $I\subset \widehat{A}$ be the ideal of $N$.
Let $J\subset I$ be the ideal of functions whose derivatives 
at points of $N$ in the directions
of $W_{1,N}$ vanish. 
On each of the quotient pieces 
$A_n:= \widehat{A}/ {\mathfrak m}^{n+1}$ define the weight filtration by 
$$
W_{-k}A_n = \sum _{q+2r \geq k} I^q J^r A_n.
$$
Consider the decreasing filtration $U= \{ U^d\}$ of $A_n$
defined by $U^d = {\mathfrak m}^d$. The filtration $W_{\cdot}$ induces
a filtration on ${\rm gr}^d_U$ and $A_n$ is a successive extension
of these filtered objects in a family over $\pp^1$. 
Furthermore,
$$
{\rm Sym}^d {\rm gr}^1_U \stackrel{\cong}{\longrightarrow}
{\rm gr}^d_U.
$$
We claim that this expression is an isomorphism of filtered bundles over 
$\pp^1$. 

To see this, work in local coordinates
$(\lambda , x_1,\ldots , x_a,y_1,\ldots , y_b,z_1,\ldots , z_c)$
such that 
$$
{\mathfrak m} = (x_1,\ldots , x_a, y_1,\ldots , y_b,z_1,\ldots , z_c),
\;\;\; 
I = (y_1,\ldots , y_b,z_1,\ldots , z_c),
$$
and 
$J = I^2 +(z_1,\ldots , z_c)$.
In these coordinates, 
$$
W_{0,N} = \langle  \frac{\partial}{\partial x_i} \rangle 
\;\;\; \mbox{ and } \;\;\; 
W_{1,N} = \langle  \frac{\partial}{\partial x_i},  \frac{\partial}{\partial y_j}\rangle  .
$$
Roughly speaking, the coordinates $x_i$ have weight $0$, the coordinates $y_j$ 
have weight $-1$, and the $z_k$ have weight $-2$. 
Then $W_{-k}A_n$ is generated by monomials 
$x^P y^Q z^R$
where 
$$
P=(p_1,\ldots , p_a), \;\; 
Q=(q_1,\ldots , q_b), \;\; 
R=(r_1,\ldots , r_c)
$$
with
$\sum q_i + 2 \sum r_j \geq k$ and $\sum p_i + \sum q_i + \sum r_ i \leq n$. 
Similarly, $U^d$ is spanned by monomials
of total degree $\geq d$. 

Locally, $W_{-k}+ U^d$ is spanned by 
monomials $x^P y^Q z^R$ for $P,Q,R$ as above, with
\newline
$\sum q_i + 2 \sum r_j \geq k$ and  $\sum p_i + \sum q_i + \sum r_ i \geq d$.
The filtrations $W$ and $U$ intersect strictly
relative to the $\lambda \in \pp^1$ coordinate, so 
the filtrations commute. Thus
$$
W_{-k}({\rm gr}^d_U) =
W_{-k}U^d / W_{-k} U^{d+1}
$$
and this a locally free sheaf over $\pp^1$ with
basis given  by monomials of the form 
$x^P y^Q z^R$ for 
$\sum q_i + 2 \sum r_j \geq k$ and  $\sum p_i + \sum q_i + \sum r_ i = d$.
This is the same expression as for the induced filtration on 
${\rm Sym}^d {\rm gr}^1_U$, showing that the muliplication map
is an isomorphism of filtered objects. 

Now, the mixed twistor property of ${\rm gr}^1_U$, which is
equivalent to the mixed twistor property of Theorem \ref{twistorproperty} 
for the weight filtration on
$ T_{\rho}M/\pp^1$, implies the mixed twistor property for 
the symmetric powers, and hence for 
${\rm gr}^d_U$. 

As the $W$-filtered bundle 
$A_n$ is a successive extension of the $W$-filtered bundles ${\rm gr}^d_U$,
this implies that $A_n$ is a mixed twistor structure. This completes the proof of
Theorem \ref{maintheorem}.

\bigskip

\noindent
{\bf Acknowledgements:}
I would like to thank the many people have contributed significantly to this work by their research articles and illuminating discussions. 
This work was supported by a grant from the Institute for Advanced Study, by
the Agence Nationale de la Recherche program 
3ia Côte d'Azur ANR-19-P3IA-0002,  and by the European Research Council Horizons 2020 grant 670624 (Mai Gehrke's DuaLL project).  This work was completed during a visit to the University of Pennsylvania
supported by the Simons Collaboration on Homological Mirror Symmetry.

\

\noindent
Université Côte d'Azur, CNRS, LJAD, France

\noindent
\verb+carlos.simpson@univ-cotedazur.fr+

\end{document}